\documentclass[11pt,reqno]{amsart}
\usepackage[T1]{fontenc}
\usepackage{lmodern}
\usepackage{microtype}
\usepackage{amsmath,amssymb,amsthm,mathtools}
\usepackage{booktabs,array,longtable}
\usepackage[hidelinks]{hyperref}
\usepackage{enumitem}
\usepackage{xcolor}
\usepackage{graphicx}
\usepackage{needspace}
\hypersetup{pdftitle={Transitive Sets of Mutually Orthogonal Latin Squares},pdfauthor={Amadou Keita and Ilya Shapiro}}
\newtheorem{theorem}{Theorem}[section]
\newtheorem{proposition}[theorem]{Proposition}
\newtheorem{lemma}[theorem]{Lemma}
\newtheorem{corollary}[theorem]{Corollary}
\theoremstyle{definition}
\newtheorem{definition}[theorem]{Definition}
\newtheorem{example}[theorem]{Example}
\theoremstyle{remark}
\newtheorem{remark}[theorem]{Remark}
\DeclareMathOperator{\Aut}{Aut}

\DeclareMathOperator{\Syl}{Syl}

\newcommand{\OA}{\mathrm{OA}}

\newcommand{\F}{\mathbb F}

\newcommand{\volp}{v_p}
\newcommand{\Defeq}{\mathrel{:=}}
\title[Transitive sets of MOLS]{Transitive Sets of Mutually Orthogonal Latin Squares}
\author{Amadou Keita}
\author{Ilya Shapiro}
\date{November 2026}
\begin{document}
\begin{abstract}
Write $n=\prod_{i=1}^u p_i^{a_i}$ for the prime factorization of $n$.
We prove that a net of order $n$ admitting a point-transitive group of
automorphisms preserving each parallel class has at most
$\min_i(p_i^{a_i}+1)$ parallel classes. This extends Jungnickel's
sharp bound for translation nets. Equivalently, the maximum number of
jointly cell-transitive mutually orthogonal Latin squares of order $n$
is $\min_i(p_i^{a_i}-1)$.  An explicit Latin square of order nine answers
affirmatively Krotov and Potapov's question about the existence of an
isotopically transitive, non-topolinear Latin square.
\end{abstract}
\maketitle

\section{Introduction}
\label{sec:intro}
A Latin square $L$ of order $n$ is an $n\times n$ array in which every symbol
from a set of $n$ symbols appears once in each row and column. Two Latin
squares are \emph{orthogonal} if all ordered pairs of their symbols occur
once upon superposition. We write $N(n)$ for the largest number of
mutually orthogonal Latin squares (MOLS) of order $n$.

If $p_1^{a_1},\ldots,p_u^{a_u}$ are the maximal prime powers dividing $n$,
the classical MacNeish--Mann construction \cite{MacNeish1922,Mann1942}
provides
\begin{equation}\label{eq:f}
  N(n)\ge f(n)\Defeq\min_{p^a\parallel n}(p^a-1).
\end{equation}
Here and throughout, $p^a\parallel n$ means that $p$ is prime,
$p^a\mid n$, and $p^{a+1}\nmid n$; thus $p^a$ is the largest power
of $p$ dividing $n$. For any positive integer $t$, $v_p(t)$ denotes
the exponent of $p$ in the prime factorization of $t$; equivalently,
$p^{v_p(t)}\parallel t$.
MacNeish's proposed equality fails already at order $21$: Parker
constructed three MOLS \cite{Parker1959}, and Bose and Shrikhande
subsequently obtained four \cite{BoseShrikhande1960}, whereas
$f(21)=2$. See \cite{BoseShrikhandeParker1960} for further
counterexamples.

There is a substantial classical literature on MOLS carrying symmetry
of their \emph{cells}.  A family of $q$ MOLS determines a net of order
$n$ and degree $r=q+2$, whose $n^2$ points are the cells and whose $r$
parallel classes are the coordinate fibers.  A \emph{translation net}
possesses a group of automorphisms acting regularly on the points and
preserving every parallel class.  Equivalently, its MOLS admit a subgroup
of the autotopy group acting regularly on the cells.  Algebraically,
translation nets are described by \emph{partial congruence partitions}:
$r$ subgroups of order $n$ in a group of order $n^2$ that intersect
pairwise trivially \cite{Jungnickel1981Nets,Hachenberger1992}.

Jungnickel \cite{Jungnickel1989} proved that a translation net of order
$n$ and degree $r$ satisfies
\begin{equation}\label{eq:Jungnickel}
 r\le\min_{p^a\parallel n}(p^a+1)=f(n)+2,
\end{equation}
with equality attainable at every order. We ask whether the same bound
holds when the group acts transitively on points and preserves each
parallel class.

The relevant weaker symmetry is \emph{cell transitivity}: the full
autotopy group is transitive on the $n^2$ cells but need not contain a
regular subgroup. In the associated net this is point transitivity
under automorphisms preserving each parallel class individually;
permutations of the parallel classes are not allowed. This homogeneity
condition has an independent history. Bailey \cite{Bailey1982}
considered Latin squares as sets of cells with three invariant block
systems and investigated a stronger automorphism-group homogeneity
condition. In the language of maximum-distance-separable (MDS) codes,
cell transitivity is called \emph{isotopic transitivity}: each autotopy
acts through independent permutations of the named coordinate
alphabets. Potapov \cite{Potapov2015MOLS} explicitly studied
isotopically transitive pairs of MOLS. For a single square, Krotov and
Potapov use \emph{topolinear} for a square admitting a regular autotopy group and ask whether
an isotopically transitive square can fail to be topolinear
\cite[Section~8]{KrotovPotapov2016}. Example~\ref{ex:9} answers their
question affirmatively at order nine. Thus cell transitivity genuinely
enlarges the class covered by Jungnickel's theorem.

\Needspace{9\baselineskip}
\begin{theorem}[Point-transitive extension of Jungnickel's bound]
\label{thm:mainintro}
Let $n\ge2$.  If an order-$n$ net with $r\ge3$ parallel classes
admits a point-transitive automorphism group preserving each parallel
class individually, then
\[
  \boxed{\quad r\le\min_{p^a\parallel n}(p^a+1).\quad}
\]
The bound is sharp, even for translation nets. Equivalently, the maximum
number of cell-transitive MOLS of order $n$ is
\[
  \boxed{\quad N_{\mathrm{tr}}(n)=\min_{p^a\parallel n}(p^a-1).\quad}
\]
\end{theorem}

To prove the bound, fix a transitive group $G$ of autotopies, a cell,
and a prime divisor $p$ of $n$. The coordinate stabilizers $H_i$ have common
pairwise intersection equal to the cell stabilizer $K$, and
$[G:H_i]=[H_i:K]=n$. If $p\nmid |K|$, counting subgroups inside a
Sylow $p$-subgroup of $G$ bounds the number of coordinates
(Proposition~\ref{prop:pfree}). For arbitrary $K$, take
$Q\in\Syl_p(K)$. The coordinate-pair bijections restrict to the
$Q$-fixed sets, and the Sylow-normalizer lemma
\cite[Theorem~7.3.2]{Robinson1996} gives a transitive action of $N_G(Q)$
on the fixed cells. Proposition~\ref{prop:descent} shows that the
resulting array has order $t$ with $1\le v_p(t)\le v_p(n)$ and that
$N_G(Q)/Q$ acts with cell stabilizer of order prime to $p$. The counting
argument therefore applies with the same number of coordinates.

Other uses of the word \emph{transitive} in the literature must be
distinguished from the coordinatewise isotopic transitivity above.
In \cite{JagerEtAl2024} it concerns transitivity on the \emph{set of
squares}; the regularity studied by Jungnickel in
\cite{Jungnickel1980} concerns symbols and row permutations rather
than regularity on all $n^2$ cells.  The terminology \emph{transitive
direction} in \cite{HachenbergerJungnickel1990} is also different from
our point-transitive hypothesis.  The distinction is important: the
regular-square constructions of \cite{Jungnickel1980} can exceed
$f(n)$ without contradicting Theorem~\ref{thm:mainintro}.

\emph{Organization of the paper:}
Section~\ref{sec:array} sets up the array terminology and
Section~\ref{sec:packet} recalls the group-packet correspondence.
Section~\ref{sec:translation} compares the setting with translation
nets, while Section~\ref{sec:mainbound} proves the sharp bound.
Section~\ref{sec:examples} concludes with a verified example demonstrating
that cell transitivity does not imply a regular cell action.

\section{Orthogonal arrays, isotopies and transitivity}
\label{sec:array}
Throughout, $n\ge2$ and $r=q+2\ge3$; in particular, we consider at least one Latin square. The condition $r\ge3$ is used in the fixed-point argument.

\begin{definition}\label{def:array}
An \emph{orthogonal array of strength two and index one} with $r$ named coordinate classes and order $n$ consists of a finite set $S$, sets $X_1,\ldots,X_r$ with $|X_i|=n$, and maps $\pi_i:S\to X_i$ such that
\[
\pi_{ij}=(\pi_i,\pi_j):S\xrightarrow{\ \sim\ }X_i\times X_j
\quad(i\ne j)
\]
are bijections. Thus $|S|=n^2$. We write $\mathcal A=(S;(X_i,\pi_i)_{i=1}^r)$, or $S\subseteq\prod_i X_i$ using the coordinate embedding. In design-theoretic notation this is an $\OA(n^2,r,n,2)$.
\end{definition}

Choosing the first two coordinates as row and column labels, define the $j$th square for $3\le j\le r$ by
\[
L_j(x_1,x_2)=\pi_j\bigl(\pi_{12}^{-1}(x_1,x_2)\bigr).
\]
The bijectivity of $\pi_{1j}$ and $\pi_{2j}$ gives the Latin property; bijectivity of $\pi_{ij}$ for $i,j\ge3$ gives mutual orthogonality. Conversely, any $q$ MOLS produces such an array. This is the usual OA--MOLS correspondence \cite{Bush1952,DeneKeedwell2015}.

\begin{definition}\label{def:isotopy}
An \emph{isotopy} $\mathcal A\to\mathcal A'$ is a bijection $u:S\to S'$ together with bijections $u_i:X_i\to X_i'$ satisfying $\pi_i'u=u_i\pi_i$. An \emph{autotopy} is an isotopy from $\mathcal A$ to itself. The autotopy group $\Aut(\mathcal A)$ acts on $S$ and all $X_i$.
\end{definition}

\begin{lemma}\label{lem:faithful}
The action of $\Aut(\mathcal A)$ on $S$ is faithful. Each pair $u_i,u_j$ with $i\ne j$ determines an autotopy uniquely.
\end{lemma}
\begin{proof}
If an autotopy fixes $S$ pointwise, then each induced coordinate permutation fixes every $\pi_i(s)$, and $\pi_i$ is surjective, so every coordinate permutation is the identity. For the second assertion use the bijection $\pi_{ij}$ to recover $u$ from $u_i\times u_j$, and then recover all remaining $u_k$ by surjectivity of $\pi_k$.
\end{proof}

\begin{definition}\label{def:transitive}
We call $\mathcal A$ (and its associated MOLS) \emph{transitive} if $\Aut(\mathcal A)$ is transitive on $S$. It is \emph{simply transitive} if $\Aut(\mathcal A)$ contains a subgroup acting freely and transitively on $S$. A subgroup $G\le\Aut(\mathcal A)$ chosen for a transitive action need not be the full autotopy group.
\end{definition}

An isotopy transports the autotopy group by conjugation and therefore preserves these properties. No interchange of coordinate indices is allowed in the definition; such interchanges belong to the larger notion of paratopy or main-class equivalence.

\begin{proposition}[Products]\label{prop:products}
Given transitive arrays $\mathcal A_1,\mathcal A_2$ with the same number $r$ of coordinate classes and respective orders $n_1,n_2$, their coordinatewise product $\mathcal A_1\times\mathcal A_2$ is a transitive array of order $n_1n_2$. If both arrays are simply transitive, so is their product.
\end{proposition}
\begin{proof}
For every $i\ne j$, the pair projection on the product is the product of the corresponding bijections, hence again a bijection. Transitive (respectively regular) subgroups of the two autotopy groups give a transitive (respectively regular) product subgroup on $S_1\times S_2$.
\end{proof}

\section{Group packets and the correspondence}
\label{sec:packet}
The group-packet description is the group-theoretic foundation for the bound and also for the examples.

\begin{definition}\label{def:packet}
A \emph{group packet of order $n$ with $r$ coordinates} is a finite group $G$ equipped with indexed subgroups $H_1,\ldots,H_r$ and a subgroup $K$ such that
\begin{equation}\label{eq:packet}
 H_i\cap H_j=K\quad(i\ne j),\qquad
 [G:H_i]=[H_i:K]=n.
\end{equation}
It is \emph{disjoint} if $K=1$.
\end{definition}

In particular, $K\le H_i$, $|H_i|=n|K|$, and
\begin{equation}\label{eq:orders}
 |G|=n^2|K|,\qquad H_iH_j=G\quad(i\ne j).
\end{equation}
The last equality follows by counting: $|H_iH_j|=|H_i||H_j|/|H_i\cap H_j|=|G|$. Neither $H_i$ nor $K$ is assumed normal.

\begin{proposition}[Packets and transitive arrays]\label{prop:packetarray}
Let $r\ge3$.
\begin{enumerate}[label=\textup{(\alph*)},leftmargin=*]
\item If a finite group $G$ acts transitively on an orthogonal array $\mathcal A$ through a homomorphism $G\to\Aut(\mathcal A)$, choose $s\in S$ and put $x_i=\pi_i(s)$, $K=G_s$, $H_i=G_{x_i}$. Then $(G;(H_i);K)$ is a group packet, and there is a coordinatewise $G$-equivariant isotopy
\begin{equation}\label{eq:orbitiso}
 (G/K;G/H_i)\xrightarrow{\ \sim\ } (S;X_i),\quad
 gK\mapsto gs,\quad gH_i\mapsto gx_i.
\end{equation}
\item Conversely, every group packet defines a transitive array
\begin{equation}\label{eq:cosetarray}
 S=G/K\longrightarrow\prod_{i=1}^rG/H_i,
 \qquad gK\longmapsto(gH_i)_{i=1}^r.
\end{equation}
\item The arrays arising from disjoint packets are simply transitive. Conversely, every simply transitive array arises from a disjoint packet.
\end{enumerate}
\end{proposition}
\begin{proof}
For (a), the $G$-equivariant map $S\to X_i$ is surjective, and transitivity on $S$ implies transitivity on $X_i$. Thus $[G:H_i]=n$. Since $\pi_{ij}$ is a $G$-equivariant bijection, its stabilizer at $(x_i,x_j)$ equals $G_s$, giving $H_i\cap H_j=K$. Orbit--stabilizer then yields $[H_i:K]=n$ and all maps in \eqref{eq:orbitiso} are bijections.

For (b), the map $G/K\to G/H_i\times G/H_j$ is injective since $H_i\cap H_j=K$; both sets have $n^2$ elements, so it is bijective. Left multiplication by $G$ is transitive on $G/K$ and induces coordinate permutations, giving a transitive array. Finally, if $K=1$ the left action on $G/K=G$ is regular. If a regular subgroup $G$ acts on an array, its point stabilizer is trivial and part (a) produces a disjoint packet.
\end{proof}

An \emph{admissible morphism} of group packets
$\alpha:(G;H_i;K)\to(G';H_i';K')$ is a group homomorphism taking
$H_i$ into $H_i'$ and inducing bijections $G/H_i\to G'/H_i'$ for all $i$.
Two group packets are called \emph{equivalent} if there is a third group
packet and admissible morphisms
\[
  (G;H_i;K)\longrightarrow (G'';H_i'';K'')
  \longleftarrow (G';H_i';K').
\]
Such a diagram is called a \emph{common-target cospan}.
Proposition~\ref{prop:equiv} shows that this defines an equivalence relation
and identifies its classes with isotopy classes of transitive arrays.

\begin{proposition}[Group-packet correspondence]\label{prop:equiv}
Two packets give isotopic arrays if and only if they admit admissible morphisms to a common packet (a common-target cospan). Consequently isotopy classes of transitive arrays with $r$ named coordinate sets correspond bijectively to packet-equivalence classes; simply transitive arrays correspond to classes containing a disjoint packet.
\end{proposition}
\begin{proof}
An admissible morphism $\alpha$ induces bijections on every coordinate coset set. For distinct $i,j$, the corresponding product is a bijection, and both pair projections in \eqref{eq:cosetarray} are bijections. Thus $G/K\to G'/K'$ is also a bijection and the resulting maps form an isotopy. Such a cospan therefore gives an isotopy.

Conversely, let $\mathcal A$ be the array from a packet $(G,H_i,K)$, and let $\widetilde G=\Aut(\mathcal A)$ be its full autotopy group. The left-translation homomorphism $G\to\widetilde G$ sends $H_i$ into the stabilizer $\widetilde H_i$ of the distinguished $i$th coordinate and induces bijections $G/H_i\to\widetilde G/\widetilde H_i$, because both sets identify with the $i$th coordinate set of $\mathcal A$. It is therefore admissible. If an isotopy $u:\mathcal A\to\mathcal A'$ does not map the chosen
base cell $s$ to the chosen base cell $s'$, transitivity of
$\Aut(\mathcal A')$ supplies an autotopy $v$ satisfying
$v(u(s))=s'$. The basepoint-preserving isotopy $v\circ u$
then identifies the two full-autotopy groups by conjugation and
carries every distinguished coordinate stabilizer to its counterpart.
Thus the canonical full-autotopy packets are isomorphic after aligning
basepoints. Applying this to the two given packets yields a common
packet, up to isomorphism, receiving admissible morphisms from both.
The final claim follows from Proposition~\ref{prop:packetarray}(c).
\end{proof}

\begin{remark}[Products of packets]
Coordinatewise products of group packets are again group packets: for fixed $r$, the packet $(G_1\times G_2,H_{1i}\times H_{2i},K_1\times K_2)$ realizes the product of arrays in Proposition~\ref{prop:products}.
\end{remark}

\begin{example}[The prime-power packet]\label{ex:field}
For a prime power $m=p^a$, let $G=\F_m^2$ under addition and let $H_i$
range over its $m+1$ one-dimensional $\F_m$-subspaces. Distinct such
subgroups intersect trivially, and $[G:H_i]=m$, so this disjoint
packet provides $m-1$ simply transitive MOLS of order $m$.
For a general order $n$, set
\[
  r_0=\min_{p^a\parallel n}(p^a+1)=f(n)+2.
\]
For each maximal prime-power factor $m=p^a$ of $n$, select exactly
$r_0$ of the $m+1$ subspaces, so that all selected packets have the
same number of coordinates. Their coordinatewise product is a disjoint
packet of order $n$ with $r_0$ coordinates, yielding precisely
$r_0-2=f(n)$ simply transitive MOLS.
\end{example}

\section{Translation nets, partial congruence partitions, and the classical bound}
\label{sec:translation}
\begin{definition}\label{def:net}
Given an orthogonal array $\mathcal A=(S;X_i,\pi_i)_{i=1}^r$,
its associated \emph{net} has point set $S$ and lines
\[
 \ell_{i,x}=\{s\in S:\pi_i(s)=x\},\qquad x\in X_i,\quad 1\le i\le r.
\]
The lines $\ell_{i,x}$ for fixed $i$ form the $i$th parallel class.
We call an automorphism \emph{direction-preserving} if it preserves
every parallel class as a set (but may permute its lines).  The net is
\emph{point-transitive} if its direction-preserving automorphism group
is transitive on $S$, and is a \emph{translation net} if it contains a
direction-preserving group acting regularly on $S$.
\end{definition}

Every parallel class partitions the $n^2$ points into $n$ lines of
size $n$, and lines belonging to different parallel classes intersect
in exactly one point. These assertions follow directly from the
pairwise bijections defining the array. Conversely a net with these
parameters determines the coordinate sets and pairwise bijections.

\begin{proposition}[The translation-net dictionary]\label{prop:translationdictionary}
For a fixed integer $r\ge3$, the following data determine the same
objects up to the appropriate coordinate-preserving isomorphisms:
\begin{enumerate}[label=\textup{(\alph*)},leftmargin=*]
\item an orthogonal array with $r$ named coordinates and a chosen
subgroup of autotopies acting regularly on its cells;
\item a net of order $n$ with $r$ named parallel classes and a chosen
translation group acting regularly on its points and preserving each
parallel class;
\item a finite group $G$ of order $n^2$ with $r$ indexed subgroups
$H_1,\ldots,H_r$ of order $n$ satisfying $H_i\cap H_j=1$ for $i\ne j$.
\end{enumerate}
The subgroups in \textup{(c)} are the components of a
\emph{partial congruence partition} (PCP), and the corresponding
$r-2$ MOLS are simply transitive.
\end{proposition}
\begin{proof}
The line fibers of the coordinate projections identify (a) and (b):
permuting each coordinate set is equivalent to preserving each
parallel class, and the condition of a regular group is unchanged.
Given a regular group $G$ and a base cell $s$, the stabilizers $H_i$
of its $i$th coordinate values satisfy $[G:H_i]=n$ and
$H_i\cap H_j=G_s=1$ by Proposition~\ref{prop:packetarray}.  Thus
$|G|=n^2$ and each $|H_i|=n$, proving (c).  Conversely a PCP is a
disjoint group packet; the associated coset array
$G\to\prod_iG/H_i$ has pairwise bijective projections, and left
translation by $G$ acts regularly on its points, preserving all its
coordinate classes. The constructions are inverse up to choices of
base cell and coordinate labels.
\end{proof}

\begin{theorem}[Jungnickel, 1989; translation-net bound]
\label{thm:Jungnickel}
\cite{Jungnickel1989}
If a translation net of order $n$ has degree $r$, then
\[
 r\le\min_{p^a\parallel n}(p^a+1).
\]
This bound is attained for every $n\ge2$ by translation nets with
translation group a direct product of elementary abelian groups.
Equivalently, the maximum number of simply transitive MOLS is
$f(n)=\min_{p^a\parallel n}(p^a-1)$.
\end{theorem}

Example~\ref{ex:field} realizes equality. Translation nets also arise from
nonabelian groups \cite{Hachenberger1991,Hachenberger1992} and from
partial spreads \cite{Jungnickel1993}.

\section{The prime-power bound by normalizer reduction}
\label{sec:mainbound}
We first bound the number of coordinates when $p$ does not divide the
order of a cell stabilizer. Taking fixed points then reduces the general
case to this one.

\begin{lemma}[Intersections of conjugates]\label{lem:conjugates}
For a group packet $(G,H_i,K)$, any distinct $i,j$ and any $x,y\in G$, the intersection $xH_ix^{-1}\cap yH_jy^{-1}$ is conjugate in $G$ to $K$.
\end{lemma}
\begin{proof}
The map $G/K\to G/H_i\times G/H_j$ is bijective. Thus the pair $(xH_i,yH_j)$ equals $(gH_i,gH_j)$ for some $g\in G$. Taking stabilizers of this pair under the diagonal left $G$-action gives
\[
xH_ix^{-1}\cap yH_jy^{-1}
= gH_ig^{-1}\cap gH_jg^{-1}=gKg^{-1}.
\]
\end{proof}

\begin{proposition}[Prime-free cell stabilizer]\label{prop:pfree}
Suppose that a finite group $G$ acts transitively by autotopies on an orthogonal array of order $n$ with $r=q+2\ge3$ coordinates. Let $K$ be a cell stabilizer, and suppose $p^a\parallel n$ for a prime divisor $p$ of $n$. If $p\nmid|K|$, then
\begin{equation}\label{eq:primefree}
  q\le p^a-1.
\end{equation}
\end{proposition}
\begin{proof}
The packet from Proposition~\ref{prop:packetarray} has $|G|=|K|n^2$ and $|H_i|=|K|n$. Fix a Sylow $p$-subgroup $P\le G$; its order is $p^{2a}$. A Sylow $p$-subgroup of $H_i$ has order $p^a$ and, by Sylow conjugacy in $G$, a conjugate $P_i$ lies in $P$. The $P_i$ may require different conjugating elements for different $i$.

For $i\ne j$, the intersection $P_i\cap P_j$ is contained in the intersection of the respective conjugates of $H_i$ and $H_j$. By Lemma~\ref{lem:conjugates}, this latter intersection is conjugate to $K$. Its order is prime to $p$, hence $P_i\cap P_j=1$. The $r=q+2$ subgroups of order $p^a$ therefore have pairwise disjoint nonidentity elements inside $P$, giving
\[
1+r(p^a-1)\le |P|=p^{2a}.
\]
As $p^a\ge2$, division by $p^a-1$ yields $r\le p^a+1$, equivalently to \eqref{eq:primefree}.
\end{proof}

\begin{lemma}[Pair projections on fixed points]\label{lem:fixedOA}
Let a finite $p$-group $Q$ act by autotopies on an array $\mathcal A=(S;X_i,\pi_i)_{i=1}^r$ with $r\ge3$. Suppose that $Q$ fixes some cell $s\in S$. Then each restricted pair projection $S^Q\to X_i^Q\times X_j^Q$ $(i\ne j)$ is bijective. The sets $X_i^Q$ have a common cardinality $t\ge1$, and $|S^Q|=t^2$. For $t\ge2$ the restricted projections define an index-one orthogonal array of order $t$ with the same $r$ named coordinates.
\end{lemma}
\begin{proof}
For each $i\ne j$, the pair projection $S\to X_i\times X_j$ is a $Q$-equivariant bijection. Passing to fixed points gives a bijection
\[
  S^Q\xrightarrow{\sim}(X_i\times X_j)^Q=X_i^Q\times X_j^Q.
\]
Since $Q$ fixes $s$, each $X_i^Q$ is nonempty. As $r\ge3$, the pairwise bijections imply that all $X_i^Q$ have a common cardinality $t\ge1$. Therefore $|S^Q|=t^2$, and for $t\ge2$ the fixed-point data form an index-one orthogonal array.
\end{proof}

\begin{lemma}[Normalizer transitivity]\label{lem:normalizer}
Let a finite group $G$ act transitively on a finite set $S$, let $s\in S$ have stabilizer $K$, and let $Q$ be a Sylow $p$-subgroup of $K$, possibly trivial. Then $N=N_G(Q)$ acts transitively on $S^Q$. Moreover
\begin{equation}\label{eq:normalizercosets}
S^Q\cong N_G(Q)/N_K(Q)
\end{equation}
as $N$-sets, with the distinguished cell corresponding to the identity coset.
\end{lemma}
\begin{proof}
Identify $S=G/K$ by $gK\mapsto gs$. A coset $gK$ is fixed by $Q$ exactly when $g^{-1}Qg\le K$. Since $Q$ is Sylow in $K$, the subgroup $g^{-1}Qg$ (of order $|Q|$) is again Sylow in $K$; this also covers the case $Q=1$. Sylow conjugacy inside $K$ supplies $k\in K$ such that $k^{-1}g^{-1}Qgk=Q$, or equivalently $gk\in N_G(Q)$. Thus the fixed coset $gK$ has a representative in $N$. Every fixed coset belongs to the $N$-orbit of $K$, proving transitivity; the stabilizer of $K$ inside $N$ is $N\cap K=N_K(Q)$.
\end{proof}

\begin{remark}[Earlier occurrence]\label{rem:Robinson}
Lemma~\ref{lem:normalizer} is a standard instance of a Sylow-normalizer fixed-point principle, for example the $k=1$ specialization of \cite[Theorem 7.3.2]{Robinson1996}. We include the argument for completeness.
\end{remark}

\begin{proposition}[One-step $p$-free reduction]\label{prop:descent}
Suppose a finite group $G$ acts transitively by autotopies on an array $\mathcal A$ of order $n$ with $r\ge3$ coordinates, let $s$ have stabilizer $K$, and fix a prime divisor $p$ of $n$ with $p^a\parallel n$. Choose $Q\in\Syl_p(K)$ and set
\[
N=N_G(Q),\quad D=N_K(Q),\quad S'=S^Q,\quad X_i'=X_i^Q.
\]
There is a positive integer $t$ such that $(S';X_i')$ is an order-$t$ transitive array under $\overline G=N/Q$, with cell stabilizer $\overline K=D/Q$. If $H_i=G_{\pi_i(s)}$, its group packet has coordinate subgroups
\[
\overline H_i=(N\cap H_i)/Q,\qquad
\overline H_i\cap\overline H_j=\overline K.
\]
Moreover,
\begin{equation}\label{eq:descentbound}
 1\le \volp(t)\le a,\qquad p\nmid|\overline K|.
\end{equation}
\end{proposition}
\begin{proof}
Since $Q$ fixes $s$, Lemma~\ref{lem:fixedOA} gives bijective restricted pair projections, a common fixed-coordinate size $t\ge1$, and $|S'|=t^2$. Every nonsingleton $Q$-orbit in $X_i$ has size divisible by $p$, so
\[
 t=|X_i^Q|\equiv |X_i|=n\equiv0\pmod p.
\]
Thus $t\ge p$, and the restricted projections define an array of order $t$.

Lemma~\ref{lem:normalizer} gives $S'\cong N/D$ as $N$-sets. The normalizer $N$ preserves each $X_i'$, and $Q$ fixes $S'$ and each $X_i'$ pointwise. Hence $\overline G=N/Q$ acts transitively by autotopies on the restricted array, with stabilizer $\overline K=D/Q$ at $s$. The stabilizer of $\pi_i(s)$ is $\overline H_i=(N\cap H_i)/Q$, since $Q\le K\le H_i$. Proposition~\ref{prop:packetarray}(a) now gives the packet intersections and indices for this action.

Since $Q\le D\le K$ and $Q$ is a Sylow $p$-subgroup of $K$, it is also a Sylow $p$-subgroup of $D$. Hence $p\nmid|\overline K|$. Moreover, $N\le G$, $|G|=n^2|K|$, and $|Q|$ is the full $p$-part of $|K|$. Therefore
\[
 \volp(|\overline G|)=\volp(|N|)-\volp(|Q|)
 \le\volp(|G|)-\volp(|Q|)=2a.
\]
But $|\overline G|=t^2|\overline K|$, with $p\nmid|\overline K|$, so
$2\volp(t)=\volp(|\overline G|)\le 2a$. Together with $p\mid t$, established above, this proves \eqref{eq:descentbound}.
\end{proof}

\begin{theorem}[Point-transitive bound]\label{thm:main}
Let $\mathcal A$ be a transitive orthogonal array of order $n\ge2$ with $r=q+2\ge3$ named coordinates. For every prime divisor $p$ of $n$ with $p^a\parallel n$,
\[
  q\le p^a-1.
\]
Consequently $N_{\mathrm{tr}}(n)=f(n)$.
\end{theorem}
\begin{proof}
Choose a transitive group of autotopies $G$, a cell $s$ and its stabilizer $K$. Proposition~\ref{prop:descent} produces another transitive array, of order $t$ with $p^b\parallel t$ and $1\le b\le a$, on the same number $r=q+2$ of coordinates. The acting group $\overline G=N_G(Q)/Q$ has a point stabilizer $\overline K=N_K(Q)/Q$ of order coprime to $p$. Proposition~\ref{prop:pfree} applies to this new array and yields
\[
  q\le p^b-1\le p^a-1.
\]
Because $p$ was an arbitrary prime divisor of $n$, $q\le\min_{p^a\parallel n}(p^a-1)$. Example~\ref{ex:field} and Proposition~\ref{prop:products} construct $f(n)$ simply transitive MOLS of order $n$, hence attain this bound.
\end{proof}

\begin{corollary}[Point-transitive nets]\label{cor:pointnets}
Let $\mathcal N$ be a net of order $n$ with $r\ge3$ parallel classes.
If a group of automorphisms of $\mathcal N$ preserves each parallel
class and is transitive on points, then
\[
 r\le\min_{p^a\parallel n}(p^a+1).
\]
This coincides with Jungnickel's bound for the smaller class of
translation nets and is best possible.
\end{corollary}
\begin{proof}
Associate to $\mathcal N$ the coordinate sets of its parallel
classes as in Definition~\ref{def:net}. The hypothesis is exactly
transitivity by autotopies of the resulting array, so
Theorem~\ref{thm:main} gives $r-2\le f(n)$. Equality occurs for the
translation nets of Theorem~\ref{thm:Jungnickel}.
\end{proof}

\begin{corollary}[Group-packet form]\label{cor:packets}
Let $(G,H_1,\ldots,H_{q+2},K)$ be a group packet of order $n$, with $q\ge1$. Without a normality or faithfulness assumption, one has
\[
 q\le p^a-1\qquad\text{for each prime power }p^a\parallel n.
\]
\end{corollary}
\begin{proof}
Apply Theorem~\ref{thm:main} to the transitive array $G/K\subseteq\prod_i G/H_i$ from Proposition~\ref{prop:packetarray}.
\end{proof}

\begin{corollary}[Sharpness and some consequences]\label{cor:conseq}
The maximum number of simply transitive MOLS of order $n$ is also
$f(n)$. In particular:
\begin{enumerate}[label=\textup{(\roman*)},leftmargin=*]
\item If $n\equiv2\pmod4$, no pair of orthogonal Latin squares of
order $n$ is \emph{jointly cell-transitive}; that is, the autotopy
group preserving both squares cannot act transitively on the cells.
\item A transitive family of order $21$ has at most two squares.
\item A transitive family of order $56$ has at most six squares.
\end{enumerate}
\end{corollary}
\begin{proof}
The maximum for simply transitive MOLS follows from
Theorem~\ref{thm:Jungnickel}. It also follows from the upper bound in Theorem~\ref{thm:main} and the construction in Example~\ref{ex:field}.
If $n\equiv2\pmod4$, then $2\parallel n$ and $f(n)=1$, so any
\emph{jointly} transitive family contains at most one square; the
autotopies of an individually transitive square need not preserve an
orthogonal mate. Finally $f(21)=2$ and $f(56)=6$ give the remaining
numerical restrictions.
\end{proof}

\begin{remark}[What the proof does not assert]\label{rem:noinduction}
The theorem does not classify the groups which realize equality nor say that every such transitive family is simply transitive. Those questions remain independent of the numerical bound.
\end{remark}

\section{Strictness of the transitivity hypothesis}
\label{sec:examples}
In the terminology of Krotov and Potapov \cite[Section~8]{KrotovPotapov2016},
a non-topolinear isotopically transitive square has a transitive autotopy
group with no subgroup acting regularly on the cells. The following square
is recorded in Keita's thesis \cite[Example~5.5.11]{KeitaThesis}, which
constructs a transitive group packet for it.

\begin{example}[A transitive square without a regular autotopy subgroup]
\label{ex:9}
Consider the following Latin square of order nine:
\[
\renewcommand{\arraystretch}{0.92}
\begin{pmatrix}
0&1&2&3&4&5&6&7&8\\
1&7&4&5&6&8&2&0&3\\
2&3&6&7&8&4&0&1&5\\
3&5&7&4&0&6&1&8&2\\
4&6&8&0&3&1&5&2&7\\
5&8&0&6&1&2&7&3&4\\
6&2&3&1&5&7&8&4&0\\
7&4&1&8&2&0&3&5&6\\
8&0&5&2&7&3&4&6&1
\end{pmatrix}.
\]
The thesis proves transitivity using the explicit packet cited above.
To see why the square is not simply transitive, write $*$ for its
multiplication. The first row and column show that $0$ is an identity,
so it is a loop. But
\[
(1*1)*1=4\ne 0=1*(1*1).
\]
Thus the loop is not associative and hence is not isotopic to a group
table (a loop isotopic to a group is necessarily isomorphic to a group).
Keita's computational classification \cite[Table~5.5.1]{KeitaThesis}
shows that every simply transitive Latin square of order nine is
isotopic to a group table. Consequently, this transitive square is not
simply transitive.
\end{example}

\textbf{Acknowledgments:}
We thank Curtis Bright and Brett Stevens for helpful discussions, and Curtis Bright for detailed comments on the manuscript.

\hypertarget{ai-disclosure}{\textit{Disclosure of AI assistance:}}
During the preparation of this manuscript, the authors used OpenAI's
ChatGPT to explore how to remove their prime-free stabilizer hypothesis.
ChatGPT suggested passing to the fixed points of a Sylow subgroup of the
cell stabilizer.
The authors also used ChatGPT to check and revise the proofs and
exposition and identify the connection between the order-nine example and
a question of Krotov and Potapov. The authors verified the mathematical
arguments, references, and computational results and take responsibility
for the final manuscript.

\bibliographystyle{amsplain}
\bibliography{references}
\end{document}